\documentclass[11pt]{article}
\usepackage{geometry}                
\usepackage{graphicx}
\usepackage{amssymb}
\usepackage{epstopdf}
\usepackage{amsmath,amsthm,amscd,amssymb}
\usepackage{latexsym}
\usepackage[colorlinks,citecolor=red,pagebackref,hypertexnames=false]{hyperref}

\numberwithin{equation}{section}
\theoremstyle{plain}
\newtheorem{theorem}{Theorem}[subsection] 
\newtheorem{lemma}[theorem]{Lemma}
\newtheorem{corollary}[theorem]{Corollary}
\newtheorem{proposition}[theorem]{Proposition}

\theoremstyle{definition}
\newtheorem{definition}[theorem]{Definition}

\theoremstyle{remark}
\newtheorem{remark}[theorem]{Remark}

\newtheorem{case[theorem]}{Case}

\title{\parbox{14cm}{\centering{Geometric configurations in the ring of integers modulo $p^l$}}}

\author{David Covert, Alex Iosevich, and Jonathan Pakianathan}

\date{\today}          

\begin{document}
\maketitle

\begin{abstract} We study variants of the Erd\H os distance problem and the dot products problem 
in the setting of the integers modulo $q$, where $q = p^{\ell}$ is a power of an odd prime.
\end{abstract}

\tableofcontents

\section{Introduction}

\subsection{Distance sets} The classical Erd\H os distance problem asks for the minimal number of distinct distances determined by a finite point set in ${\mathbb R}^d$, $d \ge 2$. The continuous analog of this problem, called the Falconer distance problem, asks for the optimal threshold $s_0 > 0$ such that if the Hausdorff dimension of a compact subset of ${\mathbb R}^d$, $d \ge 2$, is greater than $s_0$, then the set of distances determined by the subset has positive Lebesgue measure. It is conjectured that a set of $N$ points in ${\mathbb R}^d$, $d \ge 2$, determines $ \gtrapprox N^{\frac{2}{d}}$ distances and, similarly, that a subset of ${\mathbb R}^d$, $d \ge 2$, of Hausdorff dimension greater than $\tfrac{d}{2}$ determines a set of distances of positive Lebesgue measure. Here, and throughout, $X \lessapprox Y$ means that for every $\epsilon>0$ there exists $C_{\epsilon}>0$ such that $X \leq C_{\epsilon}N^{\epsilon}Y$. Similarly, $X \lesssim Y$ means that there exists $C>0$ such that $X \leq CY$.  Finally, $X \gg Y$ (or equivalently, $Y \ll X$) means that $Y = o(X)$.

The Erd\H os distance problem in the Euclidean plane has recently been solved by Guth and Katz (\cite{GK10}) in two dimensions.  They show that a set of $N$ points in ${\mathbb R}^2$ has at least $c \frac{N}{\log N}$ distinct distances. For the latest developments on the Erd\H os distance problem in higher dimensions, see \cite{KT04}, \cite{SV05}, and the references contained therein.  See \cite{Erdogan05} and the references contained therein for the best known exponents for the Falconer distance problem. 

In vector spaces over finite fields, one may define for $E \subset {\mathbb F}_q^d$, 
$$ \Delta(E)=\{||x-y||: x,y \in E\},$$ where 
$$ \|x-y\|={(x_1-y_1)}^2+\dots+{(x_d-y_d)}^2,$$ and one may again ask for the smallest possible size of $\Delta(E)$ in terms of the size of $E$. While $||\cdot||$ is not a distance in the sense of metric spaces, it is still a rigid invariant in the sense that if $||x-y||=||x'-y'||$, there exists $\tau \in {\mathbb F}_q^d$ and $O \in O_d({\mathbb F}_q)$, the group of orthogonal matrices, such that $x'=Ox+\tau$ and $y'=Oy+\tau$. 

There are several issues to contend with here. First, $E$ may be the whole vector space, which would result in the rather small size for the distance set: 
$$ |\Delta(E)|={|E|}^{\frac{1}{d}}.$$ 

Another compelling consideration is that if $q$ is a prime congruent to $1\pmod{4}$, then there exists $i \in {\mathbb F}_q$ such that $i^2=-1$. This allows us to construct a set in ${\mathbb F}_q^2,$
$$ Z=\{(t,it): t \in {\mathbb F}_q\}$$ and one can readily check that 
$$ \Delta(Z)=\{0\}.$$  

The first  non-trivial result on the Erd\H os-Falconer distance problem in vector spaces over finite fields is proved by Bourgain, Katz and Tao in \cite{BKT04}. The authors get around the first mentioned obstruction by assuming that $|E| \lesssim q^{2-\epsilon}$ for some $\epsilon>0$. They get around the second mentioned obstruction by mandating that $q$ is a prime $\equiv 3\pmod{4}$. As a result they prove that 
$$ |\Delta(E)| \gtrsim {|E|}^{\frac{1}{2}+\delta},$$ where $\delta$ is a function of $\epsilon$. 

In \cite{IR07} the second listed author along with M. Rudnev went after a distance set result for general fields in arbitrary dimension with explicit exponents. In order to deal with the obstructions outlined above, they reformulated the question in analogy with the Falconer distance problem: how large does $E \subset {\mathbb F}_q^d$, $d \ge 2$, need to be to ensure that $\Delta(E)$ contains a positive proportion of the elements of ${\mathbb F}_q$. They proved that if $|E| \ge 2q^{\frac{d+1}{2}}$, then $\Delta(E)={\mathbb F}_q$ directly in line with Falconer's result (\cite{Falc86}) in the Euclidean setting  that for a set $E$ with Hausdorff dimension greater than $\frac{d+1}{2}$ the distance set is of positive measure.   At first, it seemed reasonable that the exponent $\tfrac{d+1}{2}$ may be improvable, in line with the Falconer distance conjecture described above. In \cite{HIKR}, it was shown that the exponent $\tfrac{d+1}{2}$ is best possible in {\bf odd dimensions}, at least for general finite fields. In even dimensions it is still possible that the correct exponent is $\tfrac{d}{2}$, in analogy with the Euclidean case.  In \cite{CEHIK11}, the authors take a first step in this direction by showing that if $|E| \subset {\mathbb F}_q^2$ satisfies $|E| \ge  q^{\frac{4}{3}}$, then $|\Delta(E)| \ge cq$. This is in line with Wolff's result for the Falconer conjecture in the plane which says that the Lebesgue measure of the set of distances determined by a subset of the plane of Hausdorff dimension greater than $\frac{4}{3}$ is positive. 

\subsection{Sums and Products} Let  ${\mathbb F}_q$ denote a finite field with $q$ elements, where
$q$, a power of an odd prime, is viewed as an asymptotic parameter. In
a special case when $q=p$ is a prime,  we use the notation
${\mathbb Z}_p$. Let ${\mathbb F}_q^*$ denote the multiplicative group
of ${\mathbb F}_q$. How large does $A \subset {\mathbb F}_q$ need to be
to make sure that
\[
 dA^2=\underbrace{A^2+\dots+A^2}_{d \mbox{ {\small times}
}}\supseteq{\mathbb F}^*_q?
\]

Here,
\[ 
A^2=A\cdot A= \{a \cdot a': a,a' \in A\}\;\mbox{ and }\;2A=A+A=\{a +
a': a,a' \in A\}.
\]

It was proved in \cite{Bour08} that if $d=3$ and $q$ is prime,
this conclusion holds if the number of elements $|A| \ge
Cq^{\frac{3}{4}}$, with a sufficiently large constant $C>0$. It is
reasonable to conjecture that if $|A| \ge
C_{\epsilon}q^{\frac{1}{2}+\epsilon}$, then $2A^2\supseteq{\mathbb
F}_q^*$. This
result cannot hold, especially in the setting of general finite fields
if $|A|=\sqrt{q}$ because $A$ may in fact be a subfield. See also
\cite{BGK06}, \cite{C04}, \cite{G06}, \cite{Garaev09}, \cite{HIS07},
\cite{KS07}, \cite{TV06}, \cite{V07} and the references contained
therein on recent progress related to this problem and its analogs.
For example, Glibichuk, \cite{G06}, proved that
$$ 8A \cdot B={\mathbb Z}_p, $$ $p$ prime, provided that $|A||B|>p$ and
either $A=-A$ or $A \cap (-A)=\emptyset$. Glibichuk and Konyagin,
\cite{GK06}, proved that if $A$ is subgroup of ${\mathbb Z}_p^{*}$,
and
$|A|>p^{\delta}$, $\delta>0$, then
$$ NA={\mathbb Z}_p$$ with
$$ N \ge C4^{\frac{1}{\delta}}.$$ The above-mentioned results were
achieved by methods of arithmetic combinatorics.

In \cite{HI08} and \cite{HIKR}, the authors developed a geometric approach to this problem. Instead of studying the set $dA^2$ directly, they 
investigated the dot-product set $\Pi(E)=\{x \cdot y: x,y \in E \}$, where $E \subset {\mathbb F}_q^d$. They proceeded to show that if this set is sufficiently large, than so is the dot product set $\Pi(E)$, with results for $dA^2$ following as an immediate corollary. The result thus obtained can be summarized as follows. 

\begin{theorem} \label{geom} Let $E \subset {\mathbb F}_q^d$ and
define the {\rm incidence function}
\begin{equation}\label{if} \nu(t)=\{(x,y) \in E \times E: x \cdot
y=t\}.\end{equation}

Then
\begin{equation} \label{L2} \sum_{t \in {\mathbb F}_q }\nu^2(t) \leq
{|E|}^4q^{-1}+|E|q^{2d-1} \sum_{k \neq (0,\dots,0)}
|E \cap l_k|{|\widehat{E}(k)|}^2+(q-1)q^{-1}{|E|}^2E(0, \dots, 0),
\end{equation} where
\begin{equation}\label{linek} l_k=\{tk: t \in {\mathbb
F}^*_q\}.\end{equation}

Moreover,
\begin{equation} \label{pointwise} \nu(t)={|E|}^2q^{-1}+R(t),
\end{equation} with
\begin{equation}\hspace{.15in}\left\{\begin{array}{llllll} |R(t)|
&\leq &|E|q^{\frac{d-1}{2}}, &\mbox{ for }\;t \not=0, \\ \hfill \\
|R(0)| &\leq &|E|q^{\frac{d}{2}}.\end{array}\right.\label{trb}\end{equation}
\end{theorem}

\begin{corollary} \label{L2dot} Let $E \subset {\mathbb F}_q^d$ such that
$|E|>q^{\frac{d+1}{2}}$. Then
$$ {\mathbb F}_q^{*} \subseteq \Pi(E).$$

This result cannot in general be improved in the following sense:
\renewcommand{\theenumi}{\roman{enumi}}
\begin{enumerate}  \item Whenever ${\mathbb F}_q$ is a quadratic
extension, for any
$\epsilon>0$ there exists $E \subset {\mathbb F}_q^d$ of size $
\approx q^{\frac{d+1}{2}-\epsilon}$, such that $|\Pi(E)|=o(q)$. In
particular, the set of dot products does not contain a positive
proportion of the elements of ${\mathbb F}_q$. \item For
$d=4m+3,\,m\geq0$, for any $q\gg1$ and any $t\in {\mathbb F}_q^*$,
there exists $E$ of cardinality $ \approx q^{\frac{d+1}{2}}$, such
that $t\not\in\Pi(E).$ \end{enumerate}
\end{corollary}

In the Euclidean setting, one can ask, in analogy with the Erd\H os distance problem, how many distinct dot products does a finite subset ${\mathbb R}^d$, $d \ge 2$ determine? The lattice example suggests, as it does in the case of the distance problem, that $N$ points in ${\mathbb R}^d$ determine at least $N^{\frac{2}{d}}$ distinct dot products, up to logarithmic factors. In two dimensions this problem was recently resolved by the second listed author, Oliver Roche-Newton and Misha Rudnev (\cite{IRR11}).

\subsection{The focus of this article} In this paper, we extend the considerations above to the setting of finite cyclic rings $\mathbb{Z}_{p^l}=\mathbb{Z}/p^l\mathbb{Z}$ where $p$ is a fixed odd prime. New difficulties 
arise as these rings have many nonunits and in fact zero divisors.  For example, unique factorization fails in the polynomial ring $\mathbb{Z}_{p^2}[x]$ 
as $(x-p)^2=x(x-2p)$. One reason for considering this situation is if one is interested in answering questions about sets $E \subset \mathbb{Q}^d$ of 
rational points, one can ask questions about dot product sets and distance sets for such sets and how they compare to the answers in $\mathbb{R}^d$. 
Note by scale invariance of these questions, the problem of obtaining sharp bounds for the relationship of $|\Delta(E)|$ and $|E|$ for subsets $E$ of $\mathbb{Q}^d$ would be the same as for subsets of $\mathbb{Z}^d$ as we can scale the rational points to clear their denominators without changing 
$|\Delta(E)|$ or $|E|$.
Then for any fixed prime $p$, for a high enough prime power $p^l$ the set $E \subseteq \mathbb{Z}^d$ will reduce injectively to a subset $\bar{E}$ of 
$\mathbb{Z}_{p^l}^d$ with same size "distance set" i.e., $|\Delta(E)|=|\Delta(\bar{E})|$ where $\Delta(\bar{E})$ is defined as in the finite field case. 
Thus bounds between sizes of sets and the sizes of their 
distance sets obtained over $\mathbb{Z}_{p^l}$ translate to information for sets of rational or integer points and their distance sets. Thus information on distance sets 
obtained over $\mathbb{R}$ 
and $\mathbb{Z}_{p^l}$ for large $l$ (or equivalently over the $p$-adic integers) both give apriori information for questions framed for rational or integer points. 
This is an example of the Hasse principle where facts about rational or integer points can be obtained by using the arithmetic 
completions of $\mathbb{Q}$: the real numbers and the $p$-adic numbers for all prime numbers $p$.

In this paper we concentrate on the $p$-local analysis over $\mathbb{Z}_{p^l}$, leaving questions of assembling the local to global picture 
for a later time. We provide nearly sharp bounds for the dot-product and distance problems in this setting. 

Throughout, unless otherwise noted, $p$ will denote an odd but otherwise arbitrary prime, and $q = p^{\ell}$ will be an $\ell$-th power of an odd prime.  For a ring $R$, we let $R^{\times}$ denote the set of units in $R$.

\vskip.125in

\noindent For $x \in \mathbb{Z}_q^d$, we put $\| x \| = x_1^2 + \dots + x_d^2$.  Also, given $E \subset \mathbb{Z}_q^d$, we define the \emph{distance set} as $\Delta(E) = \{ \| x - y \| : x ,y \in E\}$.
\begin{theorem}\label{Thm:distances}
Let $E \subset \mathbb{Z}_{q}^d$, where $q = p^{\ell}$.  Suppose $|E| \gg \ell (\ell + 1)q^{\frac{(2\ell -1 )d}{2 \ell} + \frac{1}{2\ell}}$.  Then,
\[
\Delta(E) \supset \mathbb{Z}_q^{\times}.
\]
\end{theorem}

\noindent Given $E \subset \mathbb{Z}_q^d$, we define the \emph{dot-product set} $\prod(E) = \{x \cdot y : x , y \in E\}$, where $x \cdot y = x_1y_1 + \dots + x_d y_d$ is the usual dot product.

\begin{theorem}\label{Thm:dotprod} Let $E \subset \mathbb{Z}_q^d$, where $q = p^{\ell}$.  Suppose $|E| \gg \ell q^{\frac{(2 \ell - 1)d}{2 \ell} + \frac{1}{2 \ell}}$.  Then,
\[
\prod(E) \supset \mathbb{Z}_q^{\times}.
\]
\end{theorem}
\begin{corollary}\label{Cor:dA^2zq} Let $A \subset \mathbb{F}_q$, where $q = p^{\ell}$.  Suppose $|A| > q^{\frac{2\ell - 1}{2\ell} + \frac{1}{2\ell d}}$.  Then, 
\[
\mathbb{Z}_q^{\times} \subset dA^2 = A\cdot A + \dots + A\cdot A.
\]
\end{corollary}
\noindent Corollary \ref{Cor:dA^2zq} follows easily from Theorem \ref{Thm:dotprod} by setting $E = A \times \dots \times A$.  Theorem \ref{Thm:dotprod} shows that there exists a constant $B = B(p, \ell) > 0 $ so that $|E| > B q^{\left(\frac{2 \ell - 1}{2 \ell}\right)d}$ implies $\prod(E) \supset \mathbb{Z}^{\times}_{p^{\ell}}$.  To contrast this result, we prove the following;

\begin{theorem}\label{Thm:examples}
For $d \geq 3$, there exists sets $E \subset \mathbb{Z}_q^d$ of size $|E| = b q^{\left(\frac{2 \ell - 1}{2 \ell}\right)d}$, where 

\[
b = \left\{
\begin{array}{cc}
1 & \text{ if } d \text{ is even}\\
 \frac{1}{\sqrt{p}}  & \text{ if } d \text{ is odd}
\end{array} \right.
\]
and yet $|\Delta(E)|=|\prod(E)| = o(p^{\ell})$. In fact $\prod(E)$ and $\Delta(E)$ contain {\bf no} units of $\mathbb{Z}_q$.
\end{theorem}

\begin{remark}
Since $|\mathbb{Z}_{p^{\ell}}^{\times}| = p^{\ell} - p^{\ell - 1} \neq o(p^{\ell})$, Theorem \ref{Thm:examples} shows that Theorem \ref{Thm:distances} 
and Theorem \ref{Thm:dotprod} are best possible up to the factor of $\frac{1}{2 \ell}$.  In particular, if we fix $p$ and $\ell$ and let $d \to \infty$, then our results are sharp asymptotically.
\end{remark}

\subsubsection{Fourier Analysis in $\mathbb{Z}_q^d$}

For a function $f : \mathbb{Z}_q^d \to \mathbb{C}$, we define the Fourier transform of $f$ as
\[
\widehat{f}(m) = q^{-d}\sum_{x \in \mathbb{Z}_q^d} f(x) \chi(-x \cdot m)
\]
where $\chi(x) = \exp(2 \pi i x/q)$.  We note that
\[
Avg(f) = \widehat{f}(0, \dots, 0) = q^{-d} \sum_{x} f(x)
\]
is the average value of $f(x)$.  Also, we have the following useful orthogonality property.
\begin{lemma}\label{Lem:orthog} Let $\chi(x) = \exp(2 \pi i x / q)$.  Then,
\[
q^{-d} \sum_{x \in \mathbb{Z}_q^d} \chi(x \cdot m) =  \left\{
\begin{array}{cc}
1 & m = (0, \dots , 0) \\
0 & otherwise
\end{array}\right.
\]
\end{lemma}
In turn, Lemma \ref{Lem:orthog} has the following consequences:
\begin{proposition}\label{Prop:Ftransform}
Let $f , g : \mathbb{Z}_q^d \to \mathbb{C}$.  Then:
\begin{align}
f(x) =  \sum_{m \in \mathbb{F}_q^d} \chi(x \cdot m) \widehat{f}(m)
\\
q^{-d} \sum_{x \in \mathbb{Z}_q^d} f(x) \overline{g(x)} = \sum_{m \in \mathbb{Z}_q^d} \widehat{f}(m) \overline{\widehat{g}(m)}
\end{align}
\end{proposition}



\section{Proof of Distance Results (Theorem \ref{Thm:distances})}

Before we proceed, we comment about the methods used throughout the paper. Considering projections from $\mathbb{Z}_{p^{\ell}} \to \mathbb{F}_p$ and using 
finite field results can often give bounds that ensure that all nonzero elements of $\mathbb{F}_p \setminus \{ 0 \}$ are achieved as distances or dot-products. However this 
translates only to knowing that representatives of units in every mod $p$ equivalence class of $\mathbb{Z}_{p^{\ell}}$ are achieved as corresponding 
distances or dot-products in the corresponding sets in $\mathbb{Z}_{p^{\ell}}$. To ensure that {\bf all} units are achieved as dot-products or distances, we rely on  establishing fundamental Fourier estimates directly for $\mathbb{Z}_{p^{\ell}}$ throughout the paper.

We will need the following additional Lemmas, whose proofs we delay until Section \ref{Sec:prelim}.

\begin{lemma}\label{Lem:spheres}
Let $d \geq 2$ and $j \in \mathbb{Z}_q^{\times}$, where $q$ is odd.  As before, set $\| x \| = x_1^2 + \dots + x_d^2$, and denote by $S_j = \{x \in \mathbb{Z}_q^d : \| x \| = j\}$ the sphere of radius $j$.  Then,
\[
|S_j| = q^{d-1} (1 + o(1)).
\]
\end{lemma}
\begin{lemma}\label{Lem:spheredecay}
Identify $S_j$ with its indicator function.  For $j \in \mathbb{Z}_q^{\times}$ with $q = p^{\ell}$, we have
\[
\sup_{m \neq \vec{0} }\left| \widehat{S}_j(m) \right| \leq \ell ( \ell + 1)q^{-\frac{d+ 2 \ell-1}{2\ell}}
\]
\end{lemma}

With these Lemmas in tow, we are ready to proceed with the proof of Theorem \ref{Thm:distances}.  For $E \subset \mathbb{Z}_q^d$, we define the incidence function $\lambda_j = |\{(x,y) \in E \times E : \| x - y\| = j\}|$.  When $j \in \mathbb{Z}_q^{\times}$, we utilize Proposition \ref{Prop:Ftransform} and write
\begin{align*}
\lambda_j &= \sum_{x, y \in \mathbb{Z}_q^d} E(x) E(y) S_j(x - y)
\\
&= \sum_{x, y , m} E(x) E(y) \widehat{S_j}(m) \chi(m \cdot (x-y))
\\
&= q^{2d} \sum_{m} \left| \widehat{E}(m) \right|^2 \widehat{S_j}(m)
\\
&= q^{2d} \left| \widehat{E}(0)\right|^2 \widehat{S_j}(0) + q^{2d} \sum_{m \neq 0} \left| \widehat{E}(m)\right|^2 \widehat{S}_j(m)
\\
&= q^{-d} |E|^2 |S_j| + R_j.
\end{align*}
Lemma \ref{Lem:spheres} immediately implies that 
\[
\lambda_j = \frac{|E|^2}{q}(1 + o(1)) + R_j.
\]
By Lemma \ref{Lem:spheredecay}, we have
\begin{align*}
|R_j| &\leq q^{2d} \ell ( \ell + 1) q^{- \frac{d + 2 \ell - 1}{2 \ell}} \sum_{m} \left| \widehat{E}(m) \right|^2
\\
&=  \ell ( \ell + 1) |E| q^{d - \frac{d + 2 \ell - 1}{2 \ell}}
\\
&= \ell ( \ell + 1) |E| q^{\frac{(d-1)( 2 \ell - 1)}{2 \ell}}.
\end{align*}
Combining these estimates, we see that
\[
\lambda_j = \frac{|E|^2}{q}(1 + o(1)) + O\left(\ell (\ell + 1) |E| q^{\frac{(d-1)( 2 \ell - 1)}{2 \ell}} \right)
\]
which is positive whenever $|E| \gg \ell (\ell + 1) q^{\frac{d(2 \ell - 1) + 1}{2 \ell}}$, as claimed.

\section{Proof of Dot-Products Results (Theorem \ref{Thm:dotprod})}

Let $\chi(x) = \exp(2 \pi i x /q)$ as before.  Suppose $0 \leq n < m \leq log_p(q)$.  We will repeatedly rely on the following observation:

\begin{align*}
\sum_{z \in \mathbb{Z}^{\times}_{p^m}} \chi(p^n z) &= \sum_{z \in \mathbb{Z}_{p^m}} \chi(p^n z) - \sum_{z \in p \mathbb{Z}_{p^m}} \chi(p^n z)= I + II
\end{align*}
Now, the sum $I$ is zero by orthogonality as $\chi(p^n \cdot )$ is a nontrivial character on $\mathbb{Z}_{p^m}$ as $n < m$.  The sum $II$ is either zero or negative, depending on whether $n+1=m$ or not.  Either way, we will use the fact that the sum
\[
\sum_{z \in \mathbb{Z}^{\times}_{p^m}} \chi( p^n z)
\]
is nonpositive when $n < m$. In words, summing a nontrivial additive character of the type we consider over the group of multiplicative 
units always yields a nonpositive result.

For $E \subset \mathbb{Z}_q^d$, we define the incidence function $\nu(t) = \{(x,y) \in E \times E : x \cdot y = t\}$, and we show that $\nu(t) > 0$ for each unit $t \in \mathbb{Z}_q^{\times}$.  We write
\begin{align*}
\nu(t) &= q^{-1} \sum_{s \in \mathbb{Z}_q} \sum_{x,y \in E} \chi\left(s(x \cdot y)\right) \chi(-st)
\\
&= \nu_{\infty}(t) + \nu_0(t) + \nu_1(t) + \dots \nu_{\ell-1}(t),
\end{align*}
where
\[
\nu_i(t) = q^{-1}\mathop{\sum_{s \in \mathbb{Z}_q}}_{val_p(s) = i} \sum_{x,y \in E} \chi\left((s (x \cdot y)\right) \chi(-st).
\]
Recall that $val_p(x) = i$ if $p^i | x$, but $p^{i+1} \not$ \hskip-0.025in$|x$, and $val_p(0) = \infty$.  It is then plain to see that $\nu_{\infty}(t) = \frac{|E|^2}{q}$.  For the other values $i = 0 , \dots ,  \ell - 1$, note $s$ can be written in the form $s = p^i \overline{s}$, where $\overline{s}$ a uniquely determined unit 
in $\mathbb{Z}_{p^{\ell-i}}^{\times}$.  Also, viewing the term $\nu_i(t)$ as a sum in the $x$-variable, applying Cauchy-Schwarz, and extending the sum over $x \in E$ to the sum over $x \in \mathbb{Z}_q^d$, we see that
\begin{align*}
|\nu_i(t)|^2 &\leq |E| q^{-2} \sum_{x \in \mathbb{Z}_q^d} \sum_{y, y' \in E} \sum_{s, s' \in \mathbb{Z}_{p^{\ell-i}}^{\times}} \chi\left(p^i (sy - s'y')x\right) \chi\left(p^i t(s' - s)\right)
\\
&\leq |E| q^{d-2} \mathop{\mathop{\sum_{y, y' \in E}}_{p^i (sy - s'y') = \vec{0}}}_{s , s' \in \mathbb{Z}_{p^{\ell-i}}^{\times}} \chi\left(p^i t(s' - s)\right)
\end{align*}
We split the last sum into parts, $I$ and $II$, where $I$ corresponds to the sum over the terms where $s = s'$, and $II$ is over the set $(s,s')$, where $s \neq s'$.  We claim that term $II$ is a nonpositive quantity.  Accepting this for a moment, we see that
\begin{align*}
I &= |E|q^{d-2} \mathop{\sum_{s \in \mathbb{Z}_{p^{\ell-i}}^{\times}}}_{p^i s(y - y') = 0} E(y) E(y')
\\
&= |E| q^{d-2} p^{\ell-i}\left( 1 - \frac{1}{p}\right) \sum_{p^i y = p^i y'} E(y) E(y')
\\
&\leq |E| q^{d-2} p^{\ell-i} \sum_{\alpha \in \mathbb{Z}_{p^{\ell - i}}} |R_{E}(\alpha)|^2,
\end{align*}
where $R_E(\alpha) = \{ y \in E : y \equiv \alpha~(mod~p^{\ell-i})\}$.  Since the Kernel $K$ of the map
\[
K : \mathbb{Z}_{q}^d \to \mathbb{Z}_{p^{\ell-i}}^d
\]
has size $p^{id}$, it follows that 
\[
\sum_{\alpha \in \mathbb{Z}_{p^{\ell - i}}} |R_{E}(\alpha)|^2 \leq p^{id} \sum_{\alpha \in \mathbb{Z}_{p^{\ell - i}} } R_E(\alpha) = |E|p^{id}.
\]
Putting everything together, since the term $II$ is nonpositive, it follows that
\begin{align*}
|\nu_i(t)|^2 \leq I \leq |E| q^{d-2} p^{\ell-i} \cdot |E| p^{id}
\end{align*}
from which it immediately follows that
\[
|\nu_i(t)| \leq |E| q^{\frac{d-1}{2} \left(1 + \frac{i}{\ell} \right)}.
\]
Therefore, for each $t \in \mathbb{Z}_q^{\times}$, we have
\[
\nu(t) = \frac{|E|^2}{q} +\underbrace{\nu_0(t) + \dots \nu_{\ell-1}(t)}_{:= R(t)},
\]
where $|R(t)| \leq \ell |E| q^{\frac{d-1}{2}\left(2 - \frac{1}{\ell} \right)}$.  Therefore, $\nu(t) > 0$ (and hence $t \in \prod(E)$) whenever we have $|E| \gg \ell q^{\left(\frac{2\ell-1}{2\ell}\right)d + \frac{1}{2\ell}}$, as claimed.  It remains, however, to show that the term $II$ appearing in the bound for $|\nu_i(t)|^2$ is  nonpositive.  Recall that
\begin{align*}
II &= |E| q^{d-2} \sum_{y, y' \in E}\mathop{\mathop{\sum_{s, s' \in \mathbb{Z}^{\times}_{p^{\ell - i}}}}_{p^i(sy - s'y') = 0}}_{{s \neq s'}} \chi\left(p^i t (s' - s)\right)
\\
&= |E| q^{d-2} \sum_{y, y' \in E} \mathop{\mathop{\sum_{p^i (b(ay - y')) = 0}}_{a, b \in \mathbb{Z}^{\times}_{p^{\ell - i}}}}_{a \neq 1} \chi\left(p^i t (b (1-a))\right).
\end{align*}
We break up the sum $II$ into two additional pieces according to whether $1-a \in \mathbb{Z}_{p^{\ell - i }}\setminus \{ 0 \}$ is a unit or not:
\begin{align*}
II_A &= |E|q^{d-2}\sum_{y, y' \in E} \mathop{\mathop{\sum_{p^i (b(ay - y')) = 0}}_{a, b \in \mathbb{Z}^{\times}_{p^{\ell - i}}}}_{1-a \in \mathbb{Z}^{\times}_{p^{\ell - i}}} \chi\left(p^i t (b (1-a))\right)
\\
II_B &= |E|q^{d-2}\sum_{y, y' \in E} \mathop{\mathop{\sum_{p^i (b(ay - y')) = 0}}_{a, b \in \mathbb{Z}^{\times}_{p^{\ell - i}}}}_{1-a \notin \mathbb{Z}^{\times}_{p^{\ell - i}}} \chi\left(p^i t (b (1-a))\right)
\end{align*}
Note that the condition $p^i b(ay - y') = 0$ implies that $ay = y'$ in $\mathbb{Z}_{p^{\ell-i}}$, since $b$ is a unit in $\mathbb{Z}_{p^{\ell - i}}$.  Thereby summing in $b$ and applying orthogonality, we get that
\begin{align*}
II_A &= |E|q^{d-2}\sum_{y, y' \in E} \mathop{\mathop{\sum_{p^i (b(ay - y')) = 0}}_{a, b \in \mathbb{Z}^{\times}_{p^{\ell - i}}}}_{1-a \in \mathbb{Z}^{\times}_{p^{\ell - i}}} \chi\left(p^i t (b (1-a))\right)
\end{align*} 
is a nonpositive real quantity as the inner sum over $b$ is the sum of a nontrivial additive character over the group of multiplicative units.  Also, note that if $a \in \mathbb{Z}_{p^{\ell - i}}^{\times}$, but $1 - a \notin \mathbb{Z}_{p^{\ell - i}}^{\times}$, then $1 - a = p^js$, for some $0 < j < \ell - i$, where $s \in \mathbb{Z}^{\times}_{p^{\ell  - i -j}}$.  Thus, we can write

\[
II_B = |E|q^{d-2} \mathop{\mathop{\sum_{y, y' \in E}}_{p^ib((1 - p^j s )y - y') = 0}}_{b \in \mathbb{Z}^{\times}_{p^\ell - {i}}}\sum_{j = 1}^{\ell - i - 1} L_j, 
\]
where we put
\begin{align*}
L_j &:= \sum_{s \in \mathbb{Z}^{\times}_{p^{\ell - i - j}}} \chi(p^{i + j}tbs).
\end{align*}
Applying orthogonality and summing in the variable $s$ (noting $tb$ is a unit), we see that $L_j$ is either $-1$ or $0$, depending on whether $\ell - i - j = 1$ or not.  Therefore, $II_B$ is a nonpositive term, and the claim, hence the proof, follows.

\subsection{Sharpness results (Proof of Theorem \ref{Thm:examples})}
\label{sec:examples} 

In this section we provide examples to prove theorem~\ref{Thm:examples} which shows that the dot product results stated in theorem~\ref{Thm:dotprod} 
are sharp.

\begin{definition}
Let $\hat{v} \cdot \hat{w} = \sum_{i=1}^d v_iw_i$ for $\hat{v}, \hat{w} \in \mathbb{F}_p^d$.

A subspace $\mathfrak{L} \subseteq \mathbb{F}_p^d$ is Lagrangian if 
$\hat{v} \cdot \hat{w} = 0$ for all $\hat{v}, \hat{w} \in \mathfrak{L}$.
\end{definition}

It was shown in \cite{HIS07} that for $d \geq 3$ odd, $\mathbb{F}_p^d$ possesses a Lagrangian subspace of dimension 
$\frac{d-1}{2}$ while for $d \geq 4$ even, $\mathbb{F}_p^d$ possesses a Lagrangian subspace of dimension 
$\frac{d}{2}$. 

Under the projection homomorphism $\pi: \mathbb{Z}_{p^{\ell}}^d \to \mathbb{F}_p^d$ for $d \geq 3$, let 
$E=\pi^{-1} (\mathfrak{L})$ i.e., the full lift of a Lagrangian subspace of $\mathbb{F}_p^d$ to $\mathbb{Z}_q^d$ where $q=p^{\ell}$.

Now note
\[
|E| = p^{(\ell-1)d}|\mathfrak{L}| = 
\left\{
\begin{array}{ll}
q^{\frac{(2\ell-1)d}{2\ell} - \frac{1}{2\ell}} & \text{ if } d \geq 3 \text{ is odd }\\
q^{\frac{(2\ell-1)d}{2\ell}} & \text{ if } d \geq 4 \text{ is even }
\end{array}
\right.
\]

However, $\prod(E) = \{ \hat{v} \cdot \hat{w} : \hat{v}, \hat{w} \in E \}$ projects under $\pi$ to $0$ in $\mathbb{F}_p$ as $\mathfrak{L}=\pi(E)$ is Lagrangian. Thus $\prod(E) \subseteq p\mathbb{Z}_{p^{\ell}} = \text{ nonunits of } \mathbb{Z}_q$. In particular, $|\prod(E)| \leq p^{\ell-1} = o(p^{\ell})$. Furthermore since $||x-y|| = x \cdot x + y \cdot y - 2 x \cdot y$ for all $x, y \in \mathbb{Z}_{p^{\ell}}^d$, we see that $\Delta(E) \subseteq p\mathbb{Z}_{p^{\ell}}$ also and so $\Delta(E)$ also is a subset of the nonunits of $\mathbb{Z}_{q}$.

Thus if we set $b=1$ when $d$ even and $b=b(p)=p^{-\frac{1}{2}}$ when $d$ odd we have proven 
Theorem~\ref{Thm:examples} and hence shown the sharpness of the dot product and distance bounds as desired.

\section{Proofs of Preliminary Results}\label{Sec:prelim}

\subsection{Gauss Sums and Related Results}
We need the following well known results.
\begin{definition}[Quadratic Gauss sums]
For positive integers $a, b, n$, we denote by $G(a,b,n)$ the following sum
\[
G(a,b,n) := \sum_{x \in \mathbb{Z}_n} \chi(ax^2 + bx).
\]
where $\chi(x) = e^{2 \pi i x /n}$.  For convenience, we denote the sum $G(a,0,n)$ by $G(a,n)$.
\end{definition}
\begin{proposition}[\cite{IK04}]\label{Prop:quadgausssum}
Let $\chi(x) = e^{2 \pi i x /n}$.  For $a \in \mathbb{Z}_n$ with $(a,n) = 1$, we have
\[
G(a,n) = \sum_{x \in \mathbb{Z}_n} \chi(a x^2) = 
\left\{
\begin{array}{lcl}
\varepsilon_n \left( \frac{a}{n} \right) \sqrt{n}  &   & n \equiv~1~mod~2  \\
0 &   & n \equiv~2~mod~4  \\
(1+i)\varepsilon_a^{-1} \left( \frac{n}{a} \right) \sqrt{n}  &   &   n \equiv~0~mod~4~\&~a \equiv ~1~mod~2
\end{array}
\right.
\]
where $\left(\frac{\cdot }{c}\right)$ denotes the Jacobi symbol and
\[
\varepsilon_n = 
\left\{
\begin{array}{lc}
1    & n \equiv 1~mod~4  \\
i     &  n \equiv 3~mod~4
\end{array}
\right.
\]
Furthermore, for general values of $a \in \mathbb{Z}_n$, we have
\[
G(a,b,n) = \left\{
\begin{array}{lc}
(a,n)G\left(\frac{a}{(a,n)}, \frac{b}{(a,n)},\frac{n}{(a,n)} \right) & (a,n) | b \\
0 & otherwise
\end{array}\right.
\]
\end{proposition}

\begin{proposition}\label{Prop:compsquare}
Suppose that $a \in \mathbb{Z}_n^{\times}$, where $n$ is odd.  Then,
\[
G(a,b,n) = G(a,n) \chi(- b^2/4a).
\]
\end{proposition}
\begin{proof}
Since $a$ is a unit, we have 
\begin{align*}
G(a,b,n) = \sum_{x \in \mathbb{Z}_n} \chi(a(x^2 + ba^{-1}x)) &= \sum_{x \in \mathbb{Z}_n} \chi\left(a (x^2 + ba^{-1}x + b^2/4a^2)\right) \chi(-b^2/4a)
\\
&= \sum_{x \in \mathbb{Z}_n} \chi(a x^2) \chi(- b^2/4a),
\end{align*}
by the change of variables $x \mapsto x - b(2a)^{-1}$.
\end{proof}

\begin{definition}[Generalized Gauss Sum]\label{Def:Gausssum}Let $\psi$ denote a Dirichlet (multiplicative and extended by zero on nonunits) character mod $n$ and $\chi_a(x) = e^{2 \pi i ax/n}$, an additive character mod $n$.  Then, we set
\[
\tau(\psi, \chi_a) = \sum_{x \in \mathbb{Z}_n} \psi(x) \chi_a(x).
\]
When $a = 1$, we simply write $\tau(\psi, \chi_1) = \tau(\psi)$.
\end{definition}
\begin{proposition}
Suppose $\psi$ is a Dirichlet mod $q$ and $(a,q) = 1$.  Then, 
\[
\tau(\psi, \chi_a) = \overline{\psi(a)} \tau(\psi).
\]
\begin{proof}
Since $\psi(a) \overline{\psi(a)} = 1$, when $(a,q) = 1$, we have
\begin{align*}
\tau(\psi, \chi_a) =  \overline{\psi(a)} \sum_{x \in \mathbb{Z}_q} \psi(ax) \chi_1(ax) = \overline{\psi(a)} \sum_{y \in \mathbb{Z}_q} \psi(y) \chi_1(y) = \overline{\psi(a)} \tau(\psi).
\end{align*}
\end{proof}
\end{proposition}
\begin{proposition}[\cite{IK04}]
Let $\psi$ denote a Dirichlet character mod $n$ which is induced by a primitive character $\psi^*$ modulo $n^*$.  Then,
\begin{equation}
\tau(\psi) = \mu\left(\frac{n}{n^*}\right) \psi^*\left(\frac{n}{n^*}\right) \tau(\psi^*).
\end{equation}
Here, $\mu$ is the M\" obius function: 
\[ \mu(n) = 
\left\{
\begin{array}{ccl}
1 & & n = 1 \\
0 & & n \text{ is not squarefree}\\
(-1)^k&  & n = p_1\dots p_k
\end{array}
\right.
\]
Furthermore, if $\psi$ is a primitive Dirichlet character modulo $n$, then,
\begin{equation}
|\tau(\psi)| \leq \sqrt{n}
\end{equation}
\end{proposition}
\begin{corollary}\label{Cor:gengausssum}
Given any Dirichlet character $\psi$, and $\chi_a(x) = e^{2 \pi i a x / n}$, we have
\[
|\tau(\psi, \chi_a)| \leq \sqrt{n}.
\]
\end{corollary}
\subsection{Proof of Lemma \ref{Lem:spheres}}

By the Chinese Remainder Theorem, it is enough to show Lemma \ref{Lem:spheres} holds for any prime power $q = p^{\ell}$.  Assuming as much, we let $\chi(x) = e^{2 \pi i x / q}$ and $j$ a fixed unit in $\mathbb{Z}_q$.  Then,
\begin{align*}
|S_j| = \sum_{x \in \mathbb{Z}_q^d} S_j(x) &= q^{-1} \sum_{s \in \mathbb{Z}_q} \sum_{x \in \mathbb{Z}_q^d} \chi(sx_1^2) \dots \chi(s x_d^2) \chi(-sj)
\\
&= q^{-1} \left( T_{\infty} + T_0 + \dots + T_{\ell - 1}\right),
\end{align*}
where 
\begin{align*}
T_i &= \mathop{\sum_{s \in \mathbb{Z}_q}}_{val_p(s) = i} \sum_{x \in \mathbb{Z}_q^d} \chi(s x_1^2) \dots \chi(s x_d^2) \chi(-sj)
\\
&= \mathop{\sum_{s \in \mathbb{Z}_q}}_{val_p(s) = i} \left(  \sum_{x \in \mathbb{Z}_q} \chi(s x^2) \right)^d \chi(-sj).
\\
&= \mathop{\sum_{s \in \mathbb{Z}_q}}_{val_p(s) = i} (G(s,q))^d \chi(-sj)
\end{align*}
It is clear that $T_{\infty} = q^d = p^{\ell d}$.  For $i = 0 , \dots , \ell - 1$, note that if $val_p(s) = i$, then $s$ can be written in the form $s = p^i s'$, where $s'$ is a uniquely determined unit mod $\mathbb{Z}_{p^{\ell - i}}^{\times}$.  Using this fact, along with Proposition \ref{Prop:quadgausssum}, we see that
\begin{align*}
T_i &= p^{id} \sum_{s \in \mathbb{Z}^{\times}_{p^{\ell - i}}} (G(s,p^{\ell - i}))^d \chi(- s j)
\\
&= p^{id} \varepsilon_{p^{\ell - i}}^d \left(p^{\ell - i} \right)^{\frac{d}{2}} \sum_{s \in \mathbb{Z}^{\times}_{p^{\ell - i}}}  \eta\left(s \right)^{d(\ell - i)} \chi(- sj)
\end{align*}
where $\eta(s) = \left( \frac{s}{p} \right)$ is the Legendre symbol.  If $d(\ell - i)$ is even, we see that
\begin{align*}
T_i &= p^{\ell \frac{d}{2} + i \frac{d}{2}} \varepsilon_{p^{\ell - i}}^d \sum_{s \in \mathbb{Z}^{\times}_{p^{\ell - i}}} \chi(- s j) 
\\
&= - p^{\ell \frac{d}{2} + i \frac{d}{2}} \varepsilon_{p^{\ell - i}}^d \sum_{s \in p\mathbb{Z}_{p^{\ell - i}}} \chi(- s j),
\end{align*}

and hence $|T_i| \leq p^{(\ell + i)\frac{d}{2}}(p^{\ell - i -1}) = p^{\ell \left( \frac{d+2}{2} \right) + i\left( \frac{d-2}{2} \right) - 1}$.  If $d(\ell - i)$ is odd, then,
\begin{align*}
T_i &= p^{(\ell + i)\frac{d}{2}} \varepsilon_{p^{\ell - i}}^d \sum_{s \in \mathbb{Z}^{\times}_{p^{\ell - i}}} \eta(s) \chi(-sj)
\\
&= p^{(\ell + i)\frac{d}{2}} \varepsilon_{p^{\ell - i}}^d  \underbrace{\left( \sum_{s \in \mathbb{Z}_{p^{\ell - i}}} \eta(s) \chi(-sj)\right.}_{\tau(\eta, \chi_{-s})} - \underbrace{\left. \sum_{s \in p\mathbb{Z}_{p^{\ell - i}}} \eta(s) \chi(-sj)\right)}_{R}.
\end{align*}
By Corollary \ref{Cor:gengausssum}, $|\tau(\eta, \chi_{-s})| \leq \sqrt{p^{\ell - i}}$.  Using a crude bound for $|R|$, we see that
\begin{align*}
|T_i| &\leq p^{(\ell + i)\frac{d}{2}} \left(p^{(\ell - i)\frac{1}{2}} + p^{\ell - i - 1}\right).
\end{align*}
Noting that $\frac{ \ell - i}{2} \leq \ell - i - 1$ for $i \leq \ell - 2$, we have shown that $|T_{i}| \leq 2 p^{\ell \frac{d+2}{2} + i \frac{d-2}{2} - 1}$ when $i = 0, \dots , \ell - 2$, and $|T_{\ell - 1}| \leq 2 p^{\ell d - \frac{d-1}{2}}$.  Altogether, our estimates show:
\[
|T_i| \leq |T_{\ell - 1}| \leq \left\{
\begin{array}{lcl}
p^{\ell d  - \frac{d}{2}}  &   & d(\ell - 1) \text{ is even}  \\
2p^{\ell d - \frac{d-1}{2}}  &   &  d(\ell - 1) \text{ is odd} 
\end{array}
\right.
\]
Thus, we have $|S_j| = q^{d-1} + q^{-1} (T_0 + \dots + T_{\ell-1})$, where
\begin{align}
|T_0 + \dots + T_{\ell - 1}| &\leq \sum_{i=0}^{\ell - 1} |T_i| \leq 
\left\{ 
\begin{array}{lcl}
\ell p^{\ell d  - \frac{d}{2}}  &   & d(\ell - 1) \text{ is even}  \\
2\ell p^{\ell d - \frac{d-1}{2}}  &   &  d(\ell - 1) \text{ is odd} 
\end{array}
\right.
\end{align}
Putting everything together, and recalling that we set $q = p^{\ell}$, we have that
\begin{align}
|S_j| &=p^{\ell (d-1)} + O\left( q^{-1}\sum_{i=0}^{\ell - 1}|T_i| \right) 
\\
&= q^{d-1} + O\left( \left\{
\begin{array}{lcc}
\ell p^{\ell (d-1) - \frac{d}{2}} & & d(\ell - 1) \text { is even} \\
\ell p^{\ell (d-1) - \frac{d-1}{2}} & & d(\ell - 1) \text{ is odd}
\end{array} \right\}
 \right) \nonumber
\\
&= p^{\ell(d-1)} (1 + o(1)).
\end{align}
The general case follows from the Chinese Remainder Theorem.  Recall that if $q = p_1^{\ell_1} \dots p_k^{\ell_k}$, then 
\[
\mathbb{Z}_q \cong \mathbb{Z}_{p_1^{\ell_1}} \times \dots \times \mathbb{Z}_{p_k^{\ell_k}}
\]
\[
\mathbb{Z}_q^{\times} \cong \mathbb{Z}_{p_1^{\ell_1}}^{\times} \times \dots \times  \mathbb{Z}_{p_k^{\ell_k}}^{\times}
\]
Write $t \in \mathbb{Z}_q^{\times}$ as $t = (t_1, \dots t_k)$, where $t_i \in \mathbb{Z}^{\times}_{p_i^{\ell_i}}$.  To find the solutions to $\| x \| = t$ in $\mathbb{Z}_q$, one must solve the equation $\| x \| = t_i $ in each component $\mathbb{Z}_{p_i^{\ell_i}}$.  The number of solutions in $\mathbb{Z}_q$ is then the product of the number of solutions in $\mathbb{Z}_{p_i}^{\ell_i}$.  Hence:
\[
|S_t| = \prod_{i = 1}^k |S_{t_i}| = q^{d-1} (1 + o(1)).
\]

\subsection{Proof of Lemma \ref{Lem:spheredecay}}

Recall that here we require $q = p^{\ell}$.  For $m \neq \vec{0}$, we have
\begin{align*}
\widehat{S}_j(m) &= q^{-d} \sum_{x \in \mathbb{Z}_q^d} S_j(x) \chi(- x \cdot m)
\\
&= q^{-d-1} \sum_{x \in \mathbb{Z}_q^d} \sum_{t \in \mathbb{Z}_q \setminus \{ 0 \}}  \chi((x_1^2 + \dots + x_d^2 - j)t) \chi(- m \cdot x)
\\
&= q^{-d-1} \sum_{t \neq 0 } \chi(-jt) \prod_{i = 1}^{d} \left( \sum_{x_i \in\mathbb{Z}_q} \chi(x_i^2 t - m_i x_i) \right)
\\
&= q^{-d-1} \sum_{t \neq 0} \chi(- j t) \prod_{i = 1}^{d} G(t, - m_i, q)
\end{align*}
By Proposition \ref{Prop:quadgausssum}, $G(t, - m_i, q) = 0$, unless $m_i \equiv 0\pmod{\gcd(t,q)}$.  Note also that $val_p(t) = \nu$ implies that $\gcd(t, q) = p^{\nu}$.  Now,

\begin{align*}
\left. \widehat{S}_j^{\nu}(m) :=  \widehat{S}_j(m) \right|_{val_p(t) = \nu}&= q^{-d-1} \sum_{val_p(t)=\nu} \chi(-jt) p^{\nu d} \prod_{i = 1}^d G\left( \frac{t}{p^{\nu}}, \frac{- m_i}{p^{\nu}}, p^{\ell - \nu} \right).
\end{align*}
For convenience, we put $u= t/p^{\nu}$ and $\mu_i = m_i/p^{\nu}$.  We note that $u$ is a unit and $\mu_i$ are integers, since $m_i \equiv 0 \pmod{p^{\nu}}$ (as otherwise the quadratic Gauss sum vanishes).  Again, we write $\mu=(\mu_1,\dots,\mu_d)$ and 
$\| \mu \| = \displaystyle\sum_{i=1}^d \frac{m_i^2}{ p^{2\nu}}$, and $\left( \frac{\cdot}{p^m} \right)$ denotes the Jacobi symbol.  Hence,
\begin{align*}
\widehat{S}_j^{\nu}(m) &= q^{-d-1} \sum_{val_p(t) = \nu} \chi( - ju) p^{\nu d} \cdot \left(\sqrt{p^{\ell - \nu}}\right)^{d} \varepsilon^d_{p^{\ell - \nu}} \cdot \chi\left( - \frac{\| \mu \|}{4 u}\right) \cdot \left( \frac{u}{p^{\ell - \nu}}\right)^d.
\end{align*}
Proposition \ref{Prop:compsquare} yields that
\begin{align*}
\widehat{S}_j(m) &= q^{-d-1} \sum_{\nu = 0}^{\ell-1} \sum_{u \in \mathbb{Z}^{\times}_{p^{\ell - \nu}}} \chi(-ju) p^{\frac{d}{2}(\ell + \nu)} \cdot \varepsilon_{p^{\ell - \nu}}^{d} \chi\left(-\frac{\| \mu \|}{4 u} \right) \left( \frac{u}{p}\right)^{(\ell - \nu)d}
\\
&= q^{-d-1} \sum_{\nu = 0}^{\ell - 1} p^{\frac{d}{2}( \ell + \nu)} \varepsilon_{p^{\ell - \nu}}^d \sum_{u \in \mathbb{Z}^{\times}_{p^{\ell - \nu}}}\chi\left(- \frac{\| \mu \|}{4 u} - j u \right) \left( \frac{u}{p} \right)^{( \ell - \nu)d}.
\end{align*}
To finish the argument, we claim that we have the bound

\begin{equation}\label{claim}
\left| \sum_{u \in \mathbb{Z}^{\times}_{p^{\beta}}} \chi\left( a u^{-1} +b u \right) \left( \frac{u}{p} \right)^{\beta d} \right| \leq (\beta + 1) p^{\frac{\beta}{2}},
\end{equation}
where $a \in \mathbb{Z}_{p^{\beta}}$ is arbitrary, $b \in \mathbb{Z}_{p^{\beta}}$ is a unit, and $\beta$ is a positive integer.  Accepting the claim for the moment, we see that

\begin{align*}
\left| \widehat{S}_j(m) \right| &\leq q^{-d-1} \sum_{\nu = 0}^{\ell - 1} ( \ell - \nu + 1) p^{\frac{d}{2}(\ell + \nu)} p^{\frac{\ell - \nu}{2}}
\\
&\leq (\ell + 1) q^{-d-1} p^{ \frac{(d+1)\ell}{2}}\sum_{\nu = 0}^{\ell - 1} p^{\left( \frac{d-1}{2} \right) \nu}
\\
&\leq (\ell + 1) q^{-d-1} q^{\frac{(d+1)\ell}{2\ell}} \ell q^{\frac{1}{\ell} \left(\left( \frac{d-1}{2} \right) ( \ell - 1)\right)}
\\
&\leq \ell ( \ell + 1) q^{- \frac{ d + 2\ell - 1}{2 \ell}}
\end{align*}
from which the result follows.  It remains to justify \eqref{claim}.  For convenience, we define the Sali\'e sum (or twisted Kloosterman sum) as
\[
S(a,b,q) = \sum_{x \in \mathbb{Z}_q^{\times}} \chi(ax^{-1} + bx) \left( \frac{x}{q} \right),
\]
and we define the Kloosterman sum as
\[
K(a,b,q) = \sum_{x \in \mathbb{Z}_q^{\times}} \chi(ax^{-1} + bx).
\]
H. Sali\'e (\cite{Salie32}) gave the bound $|S(a,b,p)| \leq 2 \sqrt{p}$ for $\gcd(a,b,p) = 1$ and $p$ an odd prime.  Note that when $q = p^{\beta}$ and $\beta$ is even, we have $S(a,b,q) = K(a,b,q)$.  A. Weil (\cite{Weil48}) provided the well known bound $|K(a,b,q)| \leq f(q) \gcd(a,b,q)^{1/2} q^{1/2}$, where $f(q)$ denotes the number of divisors of $q$.  If $\beta d$ is even, then the sum in \eqref{claim} is reduced to the Kloosterman sum $K(a,b,q)$ which has size $|K(a,b,q)| \leq (\beta + 1)p^{\beta/2}$, since $\gcd(a,b,q) = 1$ because $b$ is a unit$\pmod{q}$.  Henceforth, we may assume that $\beta$ and $d$ are odd.  We now aim to bound
\[
\sum_{x \in \mathbb{Z}_{p^{\beta}}} \chi(ax^{-1} + bx) \left( \frac{x}{p} \right)
\]
where $\beta$ is odd.  We will make use of the following result.

\begin{lemma}[\cite{IK04}]\label{Lem:IK}
Let $\chi(x) = \exp(2 \pi i x /q)$ and let $\left( \frac{\cdot}{q} \right)$ denote the Jacobi symbol.  If $\gcd(2b,q) = 1$, then
\[
S(a,b;q) = \varepsilon_q q^{\frac{1}{2}} \left( \frac{b}{q} \right) \sum_{v^2 \equiv ab\pmod{q}} \chi(2v).
\]
If $\gcd(2ab,q) = 1$, then $S(a,b,q) = 0$ unless $ab$ is a quadratic residue$\pmod{q}$.
\end{lemma}

We will apply the result with $a = - \frac{\| \mu \|}{4}$, $b = u$, and $q = p^{\beta}$.  We consider three cases in the proof of \eqref{claim}.  First, suppose that $ab$ is not a quadratic residue$\pmod{p}$.  Since this forces $\gcd(q,2ab) = 1$, Lemma \ref{Lem:IK} implies that $S(a,b;q) = 0$.  Next, if $ab$ is a nonzero quadratic residue$\pmod{p}$, then the equation $x^2 = ab\pmod{p}$ has exactly two solutions, which by Hensel's Lemma implies that $x^2 = ab\pmod{q}$ also has two solutions.  Therefore,
\[
\left| \sum_{x^2 \equiv ab\pmod{q}} \chi(2x)\right| \leq \sum_{x^2 \equiv ab \pmod{q}} 1 = 2.
\]
Lemma \ref{Lem:IK} then implies that $|S(a,b;q)| \leq 2 \sqrt{q}$.  Finally, we consider the case when $ab \equiv 0\pmod{p}$.  We recall that $b$ is a unit, so the condition $ab \equiv 0\pmod{p}$ immediately implies that $a \equiv 0\pmod{p}$.  Hence, $val_p(b) = 0$, but $val_p(a) > 0$.  Consider the function given by $h(x) = a/x + bx$ which is defined only for units $x \in \mathbb{Z}_q^{\times}$.  A direct calculation shows that $h(x) - h(y) = (x-y)(b - a/xy)$.  Since $b - a/xy$ is nonzero in $\mathbb{Z}_p$, it is a unit in $\mathbb{Z}_q$, and it follows that $val_p(h(x) - h(y)) = val_p(x-y)$, and hence $h : \mathbb{Z}_q^{\times} \to \mathbb{Z}_q^{\times}$ is a bijective map.  Let $g$ denote the inverse map of $h$.  Then, $h(x) = y$ implies
$a/x + bx = y$ or $bx^2 - yx + a = 0$.  Since $a \equiv 0\pmod{p}$, it follows that $x = 0$ and $x=y/b$ are the two solutions to $h(x) = y\pmod{p}$.  As $g(y)=x$ 
is the solution which lies in $\mathbb{Z}_q^{\times}$ it must project to the nonzero solution mod $p$, thus $g(y)=\frac{y}{b} \text{ mod } p$. Hence, $g(y) = y/b + \theta(y)$, where $\theta(y) \equiv 0\pmod{p}$ for all $y \in \mathbb{Z}_q^{\times}$.  Note that we aim to bound the sum
\[
S(a,b,q) = \sum_{x \in \mathbb{Z}_q^{\times}} \chi(h(x)) \left( \frac{x}{p} \right)
\]
where $q = p^{\beta}$ and $\beta$ is odd.  Applying the change of variables $y = h(x)$, and using that $h(x)$ is a bijection, we see that

\[
S(a,b;q) = \sum_{y \in \mathbb{Z}_q^{\times}} \left( \frac{g(y)}{p} \right) \chi(y) = \sum_{y \in \mathbb{Z}_q^{\times}} \left( \frac{y/b + \theta(y)}{p} \right) \chi(y).
\]
As $\left(\frac{z}{p} \right)$ is determined by the value $z\pmod{p}$, and since $\theta(y) = 0\pmod{p}$, it follows that

\[
S(a,b;q) = \sum_{y \in \mathbb{Z}_q^{\times}} \left( \frac{y/b}{p} \right) \chi(y).
\]
Finally, applying the change of variables $x = y/b$, we see that
\[
S(a,b;q) = \sum_{x \in \mathbb{Z}_q^{\times}} \left(\frac{x}{p} \right) \chi(bx) = \tau(\psi, \chi_b),
\]
where $\tau(\psi, \chi_b)$ is the generalized guass sum defined in Definition \ref{Def:Gausssum}.  It follows by Corollary \ref{Cor:gengausssum} that $|S(a,b,q)| = |\tau(\psi,\chi_b)| \leq \sqrt{q}$.  After considering all cases, the claim \eqref{claim} follows.


\newpage 

\begin{thebibliography}{20}

\bibitem{Bour08} J. Bourgain, {\it Sum-product theorems and exponential sum bounds in residue classes for general modulus}, C. R. Math. Acad. Sci. Paris \textbf{344} (2007), no. 6, 349-352.

\bibitem{BGK06} J. Bourgain, A. A. Glibichuk and S. V. Konyagin. {\it
Estimates for the number
of sums and products and for exponential sums in fields of prime order}. J.
London Math. Soc. (2) \textbf{73} (2006), 380--398.

\bibitem{BKT04} J. Bourgain, N. Katz, T. Tao, {\it A sum-product estimate in finite fields, and applications}, Geom. Func. Anal. \textbf{14}, 27--57, (2004).

\bibitem{C04} E. Croot. {\it Sums of the Form $1/x_1^k+\dots 1/x_n^k$ modulo
a prime}. Integers \textbf{4} (2004).

\bibitem{CEHIK11} J. Chapman, M. B. Erdo\u{g}an, D. Hart, A. Iosevich and D. Koh, {\it Pinned distance sets, k-simplices, Wolff's exponent in finite fields and sum-product estimates}, Mathematische Zeitschrift, (accepted for publication), 2011. 


\bibitem{Erdogan05} B. Erdo\u{g}an, {\it A bilinear Fourier extension theorem and applications to the distance set problem} IMRN (2006).



\bibitem{Falc86} K. Falconer, {\it On the Hausdorff dimensions of distance sets} Mathematika \textbf{32}, 206--212, (1986).

\bibitem{Garaev09} M. Garaev, {\it The sum product estimate for large subsets of prime fields.} Proceedings of the American Mathematical Society \textbf{136} (2008), pp. 2735-2739.

\bibitem{G06} A. A. Glibichuk, {\it Combinatorial properties of sets of
residues modulo a
prime and the Erd\H os-Graham problem}. Mat. Zametki \textbf{79} (2006),
384--395;
translation in: Math. Notes \textbf{79} (2006), 356--365.

\bibitem{GK06} A. Glibichuk and S. Konyagin, {\it Additive properties of
product sets in fields of prime order}. Centre de Recherches Mathematiques,
Proceedings and Lecture Notes, 2006.

\bibitem{GK10} L. Guth and N. Katz {\it On the Erd\H os distinct distances problem in the plane.}  (preprint) arXiv:1011.4105 (2010).

\bibitem{HI08} D. Hart, A. Iosevich, {\it Sums and products in finite fields: an integral geometric viewpoint}, Contemporary Mathematics: Radon transforms, geometry, and wavelets, \textbf{464}, (2008).

\bibitem{HIS07} D. Hart, A. Iosevich, S. Solymosi, {\it Sum-product estimates in finite fields via Kloosterman sums}, Int. Math. Res. Notices (2007) Vol. 2007.

\bibitem{HIKR} D.~Hart, A.~Iosevich, D.~Koh and M.~Rudnev, {\it Averages over hyperplanes, sum-product theory in vector spaces over finite fields and the Erd\H os-Falconer distance conjecture}, 
Transactions of the AMS, \textbf{363} (2011) 3255-3275.

\bibitem{IR07} A. Iosevich, M. Rudnev, {\it Erd\H {o}s distance problem in vector spaces over finite fields}.   Trans. Amer. Math. Soc.  359,  \textbf{12}, 6127--6142, (2007).

\bibitem{IRR11} A. Iosevich, Oliver Roche-Newton, M. Rudnev, {\it On an application of Guth-Katz theorem} (accepted for publication by the Math Research Letters).

\bibitem{IK04} H. Iwaniec, and E. Kowalski, {\it Analytic Number Theory}, Colloquium Publications 53 (2004).


\bibitem{KS07}  N. H. Katz and C. Y. Shen. {\it Garaev's Inequality in finite fields not of prime order.} Online  J.  Anal. Comb. \textbf{no. 3} (2008).



\bibitem{KT04} N. H. Katz and G. Tardos {\it A new entropy inequality for the Erd\H os distance problem} 
Contemp. Math. \textbf{342}, Towards a theory of geometric graphs, 119-126, Amer. Math. Soc., Providence, RI (2004). 



\bibitem{Salie32} H. Sali\'{e}. {\it U\"{b}er die Kloostermanschen summen S(u,v;q).} Math.Z., \textbf{34} (1932), 91- 109.


\bibitem{SV05} J. Solymosi and V. Vu, {\it Near  optimal bounds for the number of distinct distances in high dimensions}, Combinatorica, (2005). 


\bibitem{TV06} T. Tao and V. Vu. {\it Additive Combinatorics}. Cambridge University Press, 2006.

\bibitem{V07} V. Vu. {\it Sum-Product estimates via directed expanders.} Mathematical Research Letters \textbf{15} (2008), 375-388.

\bibitem{Weil48} A. Weil, {\it On some exponential sums.} Proc. Nat. Acad. Sci. U.S.A. \textbf{34}, (1948), 204-207.




\vskip.125in

\end{thebibliography}
\end{document}